\documentclass[a4paper,reqno]{amsart} 
\usepackage[utf8]{inputenc} 
\usepackage{upref} 
\usepackage{mathrsfs} 
\usepackage{amsmath,amssymb,amsthm} 
\usepackage{mathtools} 
\usepackage{thmtools} 
\usepackage{enumerate} 
\usepackage{upref} 
\usepackage{indentfirst} 

\usepackage{hyperref}%
\hypersetup{
hypertexnames=false, 
linkbordercolor=[rgb]{0.5, 1, 1},%
urlbordercolor=[rgb]{0.5, 1, 1},%
citebordercolor=[rgb]{0.5, 1, 0.5}}%
\usepackage{graphicx} 
\usepackage{subcaption} 
\let\limsup\varlimsup 
\let\limsup\varlimsup 
\usepackage{cite} 

\declaretheoremstyle[
headfont=\normalfont\bfseries,
bodyfont=\normalfont\itshape,
notefont=\mdseries,
notebraces={(}{)},
headindent=\parindent,
postheadspace=0.5em,
headpunct=.,
]{plain} 

\declaretheoremstyle[
headfont=\normalfont\itshape,
bodyfont=\normalfont\rmfamily,
notefont=\mdseries,
notebraces={(}{)},
headindent=\parindent,
postheadspace=0.5em,
headpunct=.,
]{remark} 

\declaretheoremstyle[
headfont=\normalfont\bfseries,
bodyfont=\normalfont\rmfamily,
notefont=\mdseries,
notebraces={(}{)},
headindent=\parindent,
postheadspace=0.5em,
headpunct=.,
]{definition} 

\numberwithin{equation}{section}     

\declaretheorem[numberwithin=section, name=Theorem]{theorem} 
\declaretheorem[sibling=theorem, name=Lemma]{lemma} 
\declaretheorem[sibling=theorem, name=Corollary]{corollary} 
\declaretheorem[sibling=theorem, style=remark, name=Remark]{remark} 
\declaretheorem[sibling=theorem, style=remark, name=Example]{example} 
\declaretheorem[sibling=theorem, style=definition, name=Definition]{definition} 

\newcommand\setA{\mathscr{A}} 
\newcommand\setAt{\tilde{\mathscr{A}}} 
\newcommand\tA{\tilde{A}} 
\newcommand\tB{\tilde{B}} 
\DeclareMathOperator{\tr}{tr} 
\DeclareMathOperator{\co}{co} 

\begin{document} 

\title[On spectrum maximizing products]{On pairs of spectrum maximizing products with distinct factor multiplicities}

\author{Victor Kozyakin} 

\thanks{This work of the author was supported by the Ministry of science and higher education of the Russian Federation project No. FSMG-2024-0048}

\address{Higher School of Modern Mathematics MIPT\\ 
1 Klimentovskiy per., 115184 Moscow} 

\email{koziakin.vs@mipt.ru} 



\date{} 

\begin{abstract}
Recently, Bochi and Laskawiec constructed an example of a set of matrices $\{A,B\}$ having two different (up to cyclic permutations of factors) spectrum maximizing products, $AABABB$ and $BBABAA$. In this paper, we identify a class of matrix sets for which the existence of at least one spectrum maximizing product with an odd number of factors automatically entails the existence of another spectrum maximizing product. Moreover, in addition to the Bochi--Laskawiec example, the number of factors of the same name (factors of the form $A$ or $B$) in these matrix products turns out to be different. The efficiency of the proposed approach is confirmed by constructing an example of a set of $2{\times}2$ matrices $\{A,B\}$ that has spectrum maximizing products of the form $BAA$ and $BBA$.
\end{abstract} 

\maketitle 
\section{Introduction}\label{S:intro} 

Let $\setA=\{A_{1},\ldots,A_{m}\}$ be a set of $m$ real $d{\times}d$ matrices and ${\|\cdot\|}$ be a norm in~${\mathbb{R}}^{d}$. With every finite ordered tuple of symbols
\[ 
  \boldsymbol{\nu}= 
  (\nu_{1},\nu_{2},\ldots,\nu_{n})\in{\{1,\ldots,m\}}^{n},\qquad n\ge1,
\] 
we associate the matrix 
\begin{equation}\label{eq:Asigma} 
  A_{\boldsymbol{\nu}}=A_{\nu_{n}}\cdots A_{\nu_{2}}A_{\nu_{1}}
\end{equation} 
and define two numerical values: 
\begin{equation}\label{eq:spectrad} 
  \rho({\setA})= \limsup_{n\to\infty}\rho_{n}({\setA}),\qquad \bar{\rho}({\setA})= \limsup_{n\to\infty}\bar{\rho}_{n}({\setA}), 
\end{equation} 
where
\[ 
  \rho_{n}({\setA})=\max_{\boldsymbol{\nu}\in{\{1,\ldots,m\}}^{n}} \|A_{\boldsymbol{\nu}}\|^{1/n},\qquad \bar{\rho}_{n}({\setA})=\max_{\boldsymbol{\nu}\in{\{1,\ldots,m\}}^{n}} {\rho(A_{\boldsymbol{\nu}})}^{1/n}, 
\] 
and $\rho(A_{\boldsymbol{\nu}})$ denotes the spectral radius of the matrix $A_{\boldsymbol{\nu}}$.

The first limit in~\eqref{eq:spectrad}, which actually does not depend on the choice of the norm~$\|\cdot\|$, was introduced by Rota and Strang~\cite{RotaStr:IM60}, and the second somewhat later by Daubechies and Lagarias~\cite{DaubLag:LAA92}. Both quantities in~\eqref{eq:spectrad} are analogues of the well-known Gelfand formula~\cite{Gelf:MatSb41:e} for the spectral radius of the matrix, and therefore the first of them was called \emph{joint}, and the second \emph{generalized spectral radius} of a set of matrices~$\setA$. For bounded sets of matrices~$\setA$ the quantities $\rho({\setA})$ and $\bar{\rho}({\setA})$ coincide with each other~\cite{BerWang:LAA92, Breuillard2022} and in this case
\begin{equation}\label{Eq-sprad} 
  \bar{\rho}_{n}({\setA})\le \bar{\rho}({\setA})=\rho({\setA})\le \rho_{n}({\setA})\qquad\forall~n. 
\end{equation} 

Traditionally, the possibility of explicitly calculating the generalized spectral radius is associated with the fulfillment of the so-called \emph{finiteness hypothesis}, which assumes that the second limit~\eqref{eq:spectrad} is always achieved at some finite value of $n$, i.e., the leftmost inequality in~\eqref{Eq-sprad} turns into an equality for some $n$. This hypothesis was put forward in~\cite{LagWang:LAA95}, but subsequently refuted~\cite{BM:JAMS02}. Later, alternative counterexamples appeared~\cite{BTV:SIAMJMA03,Koz:CDC05:e}. The first ``explicit'' counterexample to the finiteness conjecture was constructed in~\cite{HMST:AdvMath11}, and general methods for constructing counterexamples of this kind were later developed in~\cite{MorSid:JEMS13,JenPoll:ETDS17}.

Although the finiteness conjecture has been refuted in the general case, the question of whether it holds for some specific sets of matrices remains relevant. Recently, a number of significant results have been obtained in this direction~\cite{BochiLas:SAIMJMAA24, Laskawiec:LAA25, Vladimirov:ArXiv24}.
\begin{definition}
Let $A_{\boldsymbol{\nu}}$, $\boldsymbol{\nu}= 
  (\nu_{1},\nu_{2},\ldots,\nu_{n})$, be a product of the form~\eqref{eq:Asigma}. Following~\cite{GugZen:LAA08}, we say that $A_{\boldsymbol{\nu}}$ is a \emph{spectrum maximizing product} for the set of matrices~$\setA$ if
  \[ 
    \bar{\rho}({\setA})=\rho(A_{\boldsymbol{\nu}})^\frac{1}{n}. 
  \] 
\end{definition} 

When defining spectrum maximizing products, one often adds the requirement that the matrix $A_{\boldsymbol{\nu}}$ is not a power of any other (shorter) spectrum maximizing product. In this work we do not require this.

Proving that some matrix product $A_{\boldsymbol{\nu}}$ of the form~\eqref{eq:Asigma} is spectrum maximizing is, as a rule, quite difficult. The following reasoning can help here: suppose we managed to find a vector norm $\|\cdot\|$ such that for each matrix $A_{i}\in\setA$ the inequality
\begin{equation}\label{Eq-extnorm}
 \|A_{i}x\|\le\rho(A_{\boldsymbol{\nu}})^\frac{1}{n}\|x\|,\qquad x\in\mathbb{R}^{d},
\end{equation}
holds. Then, by the definition of the joint spectral radius, the inequality $\rho({\setA})\le\rho(A_{\boldsymbol{\nu}})^\frac{1}{n}$ will hold, whence, by~\eqref{Eq-sprad} it follows that
\[ 
  \rho(A_{\boldsymbol{\nu}})^\frac{1}{n}\le \bar{\rho}({\setA})=\rho({\setA})\le \rho(A_{\boldsymbol{\nu}})^\frac{1}{n}. 
\]
Thus, $\rho(A_{\boldsymbol{\nu}})^\frac{1}{n}=\bar{\rho}({\setA})$, which implies that $A_{\boldsymbol{\nu}}=A_{\nu_{n}}\cdots A_{\nu_{2}}A_{\nu_{1}}$ is the spectrum maximizing product. In this case, inequality~\eqref{Eq-extnorm} will take the following form:
\begin{equation}\label{Eq-extnorm0} 
  \|A_{i}x\|\le\bar{\rho}({\setA})\|x\|, \qquad x\in\mathbb{R}^{d}. 
\end{equation} 

A norm in which inequality~\eqref{Eq-extnorm0} holds is called \emph{extremal} for the set of matrices $\setA$. This term apparently first appeared in~\cite{Bar:ACC95}. Later, in~\cite{Bar:AIT88-2:e} extremal norms were used as a basic tool for analyzing the growth rate of matrix products. Subsequently, this approach was widely used in numerous works, among which we highlight~\cite{Wirth:LAA02}. More detailed information about the growth rate of the norms of matrix products can be obtained in the case when the norm $\|\cdot\|$ for some~$\rho$ satisfies the equality
\begin{equation}\label{Eq-mane-bar} 
  \max\left\{\|A_{0}x\|,\|A_{1}x\|,\ldots,\|A_{m-1}x\|\right\}=\rho\|x\|
\end{equation}
for all $x\in \mathbb{R}^{d}$. A norm satisfying condition~\eqref{Eq-mane-bar} is usually called the \emph{Barabanov norm} corresponding to the set of matrices $\setA$. In~\cite[Thm.~2]{Bar:AIT88-2:e} it is shown that such a norm exists for
``almost all'' sets of matrices $\setA$, but unfortunately it is defined by some computationally nonconstructive limit procedure. Moreover, the equality~\eqref{Eq-mane-bar} can be satisfied if and only if $\rho=\rho(\setA)$, and therefore in view of~\eqref{Eq-sprad} any Barabanov norm is extremal.

The impetus for this work was the article~\cite{BochiLas:SAIMJMAA24}, in which an example of a set of matrices $\{A,B\}$ was constructed that has two different (up to cyclic permutations of factors) spectrum maximizing products of matrices, $AABABB$ and $BBABAA$. Below we propose another way to construct matrix sets, for which the number of spectrum maximizing products with an odd number of factors, if any, automatically turns out to be at least two. Moreover, in addition to the example from~\cite{BochiLas:SAIMJMAA24}, the number of factors of the same name (factors $A$ or $B$) in these matrix products is different.

In our approach, we develop the idea from~\cite{PW:LAA08} (proposed and used there, but not explicitly formulated by the authors) that the very structure of the set of matrices can significantly contribute to the answer to this question. Recall that in~\cite[Prop.~18]{PW:LAA08} the so-called ``symmetric'' sets of matrices $\setA=\{A_{1},A_{2},\ldots,A_{m}\}$  were considered, which have the property that, together with each matrix $A_{i}$, this set also contains the transposed matrix $A_{i}^{\mathtt{t}}$. This (symmetry) turned out to be sufficient for the Euclidean norm to be an extremal norm for a given family of matrices, and the generalized spectral radius was achieved on one of the matrix products of the form~${A_{i}^{\mathtt{t}}A_{i}}$. 

In Section~\ref{sec:main} (Theorem~\ref{thm:smp} and Corollary~\ref{cor:smp}) we will present the details of the implementation of the corresponding idea only to the extent that will be sufficient for construction in Section~\ref{sec:example} of an example of a set of matrices $\setA=\{A,B\}$ for which the spectrum maximizing products have the forms $BAA$ and $BBA$ (up to cyclic permutations of factors). The proof that in this example the matrix products $A_{\boldsymbol{\nu}}=BAA$ or $A_{\boldsymbol{\nu}}=BBA$ are spectrum maximizing will be done by constructing an appropriate norm $\|\cdot\|$, in which these matrix products satisfy inequality~\eqref{Eq-extnorm}.

\section{Main result}\label{sec:main} 

Let's start by defining the key concept in this section. We present this concept, as well as the following Lemma~\ref{lem:Lselfsim}, in accordance with the recommendations of an anonymous reviewer, who drew our attention to the fact that the use of Friedland's theorem~\cite[Thm.~2.2]{Friedland:AM83} can significantly simplify the study of the problem of simultaneous similarity of sets of matrices.
\begin{definition}\label{def:selfsim}
A set of matrices $\setA=\{A,B\}$, $A\neq B$, will be called \emph{$\tau$-per\-mut\-able} if there exists a mapping $\tau(X):=S^{-1}XS$ that performs a simultaneous similarity transformation between the sets of matrices $\{A,B\}$ and $\{B,A\}$, i.e., such that 
\[ 
\tau(A)=B,\quad \tau(B)=A. 
\]
\end{definition} 

In this definition, the matrix $S$ is naturally assumed to be nonsingular. Obviously, a set of matrices $\setA=\{A,B\}$ can be $\tau$-permutable only in the case when the matrices $A$ and $B$ are \emph{isospectral}, i.e., their spectra coincide:
\[ 
\sigma(A)=\sigma(B). 
\] 

If in Definition~\ref{def:selfsim} we take the mapping $\tau(X):=X^{\mathtt{t}}$ as $\tau$, then the set of matrices $\setA=\{A,B\}$ will be so-called ``symmetric''. Sets of matrices of this kind, as we have already mentioned, were introduced and studied in~\cite[Prop.~18]{PW:LAA08}.

In what follows, we will focus on irreducible matrix sets consisting of a pair of real matrices $A$ and $B$ of dimension $2{\times}2$. Recall that a set of matrices is called \emph{irreducible} if its matrices do not have common invariant subspaces, with the exception of the zero space and the entire space. For such sets of matrices, the $\tau$-permutability condition becomes very simple.

\begin{lemma}\label{lem:Lselfsim}
An irreducible set of matrices $\setA=\{A,B\}$ is permutable under some similarity transformation $\tau$ if and only if
\[
\tr(A)=\tr(B),\quad \det(A)=\det(B).
\]
\end{lemma}
\begin{proof}We use the theorem of
Friedland~\cite[Thm.~2.2]{Friedland:AM83}, according to which irreducible sets $\{A,B\}$ and $\{\tilde{A},\tilde{B}\}$ of $2{\times}2$~matrices are simultaneously similar if and only if the $5$-tuples
\[
\bigl(\tr(A),~\tr(A^{2}),~\tr(B),~\tr(B^{2}),~\tr(AB)\bigr)
\]
and
\[
\bigl(\tr(\tilde{A}),~\tr(\tilde{A}^{2}),~\tr(\tilde{B}),
~\tr(\tilde{B}^{2}),~\tr(\tilde{A}\tilde{B})\bigr)
\]
coincide. Then the sets of matrices $\{A,B\}$ and $\{B,A\}$ are simultaneously similar if and only if
\[
\bigl(\tr(A),\tr(A^{2}),\tr(B),\tr(B^{2}),\tr(AB)\bigr)=
\bigl(\tr(B),\tr(B^{2}),\tr(A),\tr(A^{2}),\tr(BA)\bigr).
\]
Now let us note that $\tr(AB)=\tr(BA)$, and for $2{\times}2$~matrices the following equalities hold:
\[
\tr(A^{2})=\tr(A)^{2}-2\det(A),\quad \tr(B^{2})=\tr(B)^{2}-2\det(B).
\]
This implies that the last two $5$-tuples coincide if and only if $\tr(A)=\tr(B)$ and $\det(A)=\det(B)$. The lemma is proved.
\end{proof}

Further, it will be convenient for us to limit ourselves to considering either matrices of ``rotation with stretching-contraction along the coordinate axes'' (see Example~\ref{ex:alt}), or matrices of a more ``exotic'' type, which can also be treated as matrices of rotation with stretching-contraction along the coordinate axes, but at the same time additionally have a zero upper-diagonal element (see Example~\ref{ex:main}).

\begin{example}\label{ex:alt} 
Let
 \begin{equation}\label{eq:setA-alt}
 A=\begin{pmatrix*}[r] \cos\varphi&-\frac{1}{\varkappa}\sin\varphi\\ \varkappa\sin\varphi&\cos\varphi
 \end{pmatrix*},\quad
 B=\begin{pmatrix*}[r] \cos\varphi&-\varkappa\sin\varphi\\ \frac{1}{\varkappa}\sin\varphi&\cos\varphi
 \end{pmatrix*}
 \end{equation}
  be a pair of matrices depending on the parameters $\varphi$ and $\varkappa>1$. The set of matrices $\setA=\{A,B\}$ is irreducible only for $\varphi\neq 0,\pi$, since only in this case the matrices~\eqref{eq:setA-alt} do not have nontrivial invariant subspaces. In this case, the set $\{A,B\}$ becomes $\tau$-permutable if we take the following similarity map as $\tau$:
  \begin{equation}\label{eq:def-tau-alt} 
    \tau(X)=S^{-1}XS,\quad\text{where}\quad X\in\setA,\quad S=\begin{pmatrix*}[r] 0&-1\\ 1&0 
    \end{pmatrix*}. 
  \end{equation} 
\end{example} 

\begin{example}\label{ex:main}
Let
 \begin{equation}\label{eq:setA-main}
 A=\begin{pmatrix*}[c] 0&-\frac{1}{\varkappa}\\ \varkappa&2\cos\varphi
 \end{pmatrix*},\quad
 B=\begin{pmatrix*}[c] 0&-\varkappa\\ \frac{1}{\varkappa}&2\cos\varphi
 \end{pmatrix*}
 \end{equation}
  be a pair of matrices depending on the parameters $\varphi\neq 0,\pi$ and $\varkappa>1$. As in the previous example, for $\varphi\neq 0,\pi$ the matrices~\eqref{eq:setA-main} are similar to rotation matrices that do not have nontrivial invariant subspaces, and therefore the set of matrices $\setA=\{A,B\}$ is irreducible. This set becomes $\tau$-permutable if we take the following similarity mapping as $\tau$:
  \begin{equation}\label{eq:def-tau-main}
    \tau(X)=S^{-1}XS,\quad\text{where}\quad X\in\setA,\quad
    S=\begin{pmatrix*}[c]
      \frac{2\varkappa\cos\varphi}{\varkappa^{2}+1}&1\\[6pt]
      -1&-\frac{2\varkappa\cos\varphi}{\varkappa^{2}+1}
    \end{pmatrix*}.
  \end{equation}
\end{example} 

The reason why we pay special attention to the class of $\tau$-permutable sets of matrices is explained by the following theorem.
\begin{theorem}\label{thm:smp}
Let the set of matrices $\setA=\{A,B\}$ be $\tau$-permutable under some mapping $\tau$. Then for each product of matrices
 \begin{equation}\label{eq:prodM}
 M=M_{k}M_{k-1}\cdots M_{1},\qquad M_{i}\in\setA\quad\text{for}\quad i=1,2,\ldots,k,
 \end{equation}
the spectrum of the matrix $\tau(M)$ coincides with the spectrum of the matrix $M$.

If the number of factors $M_{i}$ in~\eqref{eq:prodM} is odd, then the matrices $M$ and $\tau(M)$ have different numbers of factors of the form $A$ (as well as of the form $B$), and therefore no cyclic permutation of factors of $M$ can coincide with any cyclic permutation of factors of $\tau(M)$.
\end{theorem} 

We will put the proof of the theorem in Section~\ref{sec:proof-smp}. We also formulate a consequence of this theorem that does not require special proof and relates to spectrum maximizing products of matrices.
\begin{corollary}\label{cor:smp}
Let, under the conditions of Theorem~\ref{thm:smp}, the product of matrices~\eqref{eq:prodM} be spectrum maximizing and consist of an odd number of factors. Then $\tau(M)$ is also a spectrum maximizing product other than $M$ (up to cyclic permutations), and the matrices $M$ and $\tau(M)$ have different numbers of factors of the same name, i.e., factors $A$ (and also $B$).
\end{corollary} 

\begin{example}\label{ex:smp-main}
Let $\setA=\{A,B\}$ be a $\tau$-permutable set of matrices from Example~\ref{ex:alt} or~\ref{ex:main}. Let it already be proven that for this set of matrices the product of matrices $BAA$ (and hence its cyclic permutations $ABA$ and $AAB$) is spectrum maximizing  for some values of the parameters $\varkappa$ and $\varphi$. Then, by Theorem~\ref{eq:prodM}, the matrix products $BBA$, $ABB$, and $BAB$ are also spectrum maximizing, and all of them differ from the matrix product~$BAA$ and its cyclic permutations.
\end{example} 

Note that Corollary~\ref{cor:smp} is conditional: it does not establish the existence of spectrum maximizing products for the set $\setA$, but only specifies the conditions in terms of properties of the set $\setA$ under which the spectrum maximizing products are not unique. Proving the existence of spectrum maximizing products for the set~$\setA$ is a separate problem, which can be solved according to the scheme outlined in the introduction, by constructing a norm $\|\cdot\|$ satisfying inequality~\eqref{Eq-extnorm}.

Of particular interest to us will be the sets of matrices~\eqref{eq:setA-alt} and~\eqref{eq:setA-main} with the parameter value $\varphi=\frac{2\pi}{3}$. This seemingly rather strange and unexpected choice of the parameter $\varphi$ is explained by the fact that numerical modeling gave reason to believe (but did not prove!) that it is with this value of the parameter $\varphi$ that the sets of matrices~\eqref{eq:setA-alt} and~\eqref{eq:setA-main} can have, for some $\varkappa>1$, spectrum maximizing products of length $3$. Moreover, below we will limit ourselves to a detailed analysis of only the set of matrices $\setA=\{A,B\}$ with matrices~\eqref{eq:setA-main}, which for $\varphi=\frac{2\pi}{3}$ take the simple form
\begin{equation}\label{eq:AB2}
 A=\begin{pmatrix*}[c] 0&-\frac{1}{\varkappa}\\ \varkappa&-1
 \end{pmatrix*},\quad
 B=\begin{pmatrix*}[c] 0&-\varkappa\\ \frac{1}{\varkappa}&-1
 \end{pmatrix*},
\end{equation}
significantly simplifying symbolic manipulations with these matrices in further considerations.

\section{Basic example}\label{sec:example}
In this section, we consider the matrix set $\setA=\{A,B\}$ from Example~\ref{ex:main}, which is $\tau$-permutable with respect to a mapping~$\tau$ of the form~\eqref{eq:def-tau-main}. Note that the mapping $\tau(X):=S^{-1}XS$ is \emph{multiplicative}, i.e.,
\[
 \tau(X_{k}X_{k-1}\cdots X_{1})= \tau(X_{k})\tau(X_{k-1})\cdots \tau(X_{1})
\]
for any set of matrices $X_{1},X_{2}\ldots,X_{k}$. Therefore
\[
\tau(BAA)=\tau(B)\tau(A)\tau(A)=ABB. 
\]
Then, by Theorem~\ref{thm:smp}, the matrices $BAA$ and $ABB$ are isospectral: $\sigma(BAA)=\sigma(ABB)$. Since the matrix $BAA$ is isospectral to the products of its cyclic permutations $AAB$ and $ABA$, and the matrix $ABB$ is isospectral to the products of its cyclic permutations $BBA$ and $BAB$, then all triple products of the matrices $A$ and $B $, other than the products $AAA$ and $BBB$, are also isospectral: 
\[ 
\sigma(BAA)=\sigma(AAB)=\sigma(ABA)=\sigma(ABB)=\sigma(BBA)=\sigma(BAB). 
\]
Hence, similar equalities also hold for the spectral radii of the indicated matrices:
\[ 
\lambda:=\rho(BAA)=\rho(AAB)=\rho(ABA)=\rho(ABB)=\rho(BBA)=\rho(BAB).	
\] 

Further we will assume that in the equalities defining the matrices $A$ and $B$ in Example~\ref{ex:main}, the value of $\varphi$ is determined by the equality $\varphi= \frac{2\pi}{3}$. In this case, the matrices $A$ and $B$ take the form~\eqref{eq:AB2} and, as is easy to calculate,
\begin{equation}\label{eq:baa-bba}
BAA 
=\begin{pmatrix*}[c]
	\varkappa^{2}&0\\[6pt]
	\varkappa-\frac{1}{\varkappa}&\frac{1}{\varkappa^{2}}
\end{pmatrix*},\qquad
BBA 
=\begin{pmatrix*}[c]
	\varkappa^{2}&\frac{1}{\varkappa}-\varkappa\\[6pt]
	0&\frac{1}{\varkappa^{2}}
\end{pmatrix*}.
\end{equation}
Both of these matrices have unit determinant and the same trace:
$\tr(BAA)=\tr(BBA)=\varkappa^{2}+\frac{1}{\varkappa^{2}}$. Therefore
\begin{equation}\label{eq:tau-lambda}
\lambda=\varkappa^{2}
\end{equation}
is their common maximum eigenvalue.

Our goal is to show in the remainder of this section that in this case the generalized spectral radius $\bar{\rho}(\setA)$ of a set of matrices~\eqref{eq:AB2} is given by the equality
\[
 \bar{\rho}(\setA)=\lambda^{1/3},
\]
and the products of matrices $BAA$ and $BBA$ with matrices $A$ and $B$ of the form~\eqref{eq:AB2} are spectrum maximizing. To do this, following the scheme outlined in the introduction, we will show that for the set of matrices~\eqref{eq:AB2}, a norm $\|\cdot\|$ can be found in which the inequalities
\begin{equation}\label{eq:Anorm}
 \|Ax\|, \|Bx\| \le \lambda^{1/3}\|x\|,\qquad x\in\mathbb{R}^{2},
\end{equation}
are satisfied. Technically, it will be easier to do this not for the set of matrices $\setA=\{A,B\}$, but for the ``normalized'' set of matrices $\setAt=\{\tA,\tB\}$, where
\begin{equation}\label{eq:def-tilde-ab}
 \tA=\frac{1}{\lambda^{1/3}}A,\qquad \tB=\frac{1}{\lambda^{1/3}}B.
\end{equation}

In terms of the matrices $\tA$ and $\tB$, condition~\eqref{eq:Anorm} takes the form 
\begin{equation}\label{eq:tAnorm} 
  \|\tA x\|, \|\tB x\| \le \|x\|,\qquad x\in\mathbb{R}^{2}. 
\end{equation}
In addition, since the matrices $A$ and $B$ of the form~\eqref{eq:AB2} satisfy the equalities $AAA = BBB = I$ (precisely in order to ensure the fulfillment of these equalities, the quantity $\varphi$ was chosen equal to $\frac{2\pi}{3}$), then
\begin{equation}\label{eq:tilde-aaa}
 \tA\tA\tA=\tB\tB\tB=\frac{1}{\lambda}I\quad\Longrightarrow\quad \rho(\tA\tA\tA)=\rho(\tB\tB\tB )=\frac{1}{\lambda},
\end{equation}
and the remaining triple products of the matrices $\tA$ and $\tB$ will have 
simple maximum eigenvalue $1$, and thus unit spectral radius:
\[
  \rho(\tB\tA\tA)=\rho(\tA\tA\tB)=\rho(\tA\tB\tA)=\rho(\tA\tB\tB)=
  \rho(\tB\tB\tA)=\rho(\tB\tA\tB)= 1.
\]
 
\subsection{Extremal norm construction}\label{ssec:norm} 
We will determine the norm $\|\cdot\|$ required in~\eqref{eq:tAnorm} by specifying the boundary
\[
 S=\{x: \|x\|=1\}
\]
of its unit ball. In addition, following the ideas developed in~\cite{GugProt:FCM13, JSRToolbox}, the set $S$ will be sought in the form of a balanced (centrally symmetric) convex dodecagon whose vertices $v_{1},v_{2},\ldots ,v_{12}$ coincide with suitably scaled plus-minus eigenvectors of the matrices
\[
 \tB\tA\tA,\quad \tA\tA\tB,\quad \tA\tB\tA,\quad \tA\tB\tB,\quad \tB\tB\tA,\quad \tB\tA\tB,
\]
corresponding to the unit eigenvalue of these matrices.

To clarify the procedure for ``scaling'' the mentioned eigenvectors chosen as 
vertices of the polygon $S$, recall that for matrices $A$ and $B$ from the 
set~\eqref{eq:AB2} the value~$\lambda$ is expressed through~$\varkappa$ by~\eqref{eq:tau-lambda}. In this case, the eigenvectors $v$ and $w$ of the matrices $\tB\tA\tA=\frac{1}{\lambda}BAA$ and $\tB\tB\tA=\frac{1}{\lambda}BBA$, respectively, corresponding to the eigenvalue $1$, by virtue of~\eqref{eq:baa-bba} have the form
\[ 
  v=\begin{pmatrix*}[c] 
    1\\[2pt] \frac{\varkappa}{1+\varkappa^{2}} 
  \end{pmatrix*},\quad 
  w=\begin{pmatrix*}[c] 
    1\\[2pt]0\vphantom{\frac{\varkappa}{1+\varkappa^{2}}} 
  \end{pmatrix*}, 
\] 
i.e.,
\begin{equation}\label{eq:eigen} 
  v = \tB\tA\tA v,\quad w=\tB\tB\tA w. 
\end{equation} 

As the first two vertices of the constructed dodecagon $S$ we take the points
\[
 v_{1}=\mu v,\quad v_{2}= w,
\]
where $\mu>0$ is the ``scaling parameter'' to be further defined, and then 
add the points
\[
 -v_{1},\quad \pm\tA v_{1},\quad \pm\tA\tA v_{1},\qquad -v_{2},\quad \pm\tA v_{2},\quad \pm\tB\tA v_{2}
\]
to the set of vertices of the constructed dodecagon. By construction, the points $\pm v_{1}$, $\pm\tA v_{1}$, $\pm\tA\tA v_{1}$ are eigenvectors of the matrix $\tB\tA\tA$ and cyclic permutations of its factors $\tA\tB\tA$ and $\tA\tA\tB$, and the points $\pm v_{2}$, $\pm\tA v_{2}$, $\pm\tB\tA v_{2}$ are eigenvectors of the matrix $\tB\tB\tA$ and cyclic permutations of its factors $\tA\tB\tB$ and $\tB\tA\tB$.

Let us introduce the following notation for the resulting $12$ points:
\begin{alignat}{6}\label{eq:s1-s6} 
 v_{1} & =\mu v,~ & v_{2} & =w,~ & v_{3} & =-\tA v_{1},~ & v_{4} & =-\tA v_{2},~ & v_{5} & =-\tA v_{3},~ & v_{6} & =-\tB v_{4}, \\ \label{eq:s7-s12} v_{7} & =-v_{1},~ & v_{8} & =-v_{2},~ & v_{9} & =-v_{3},~ & v_{10} & =-v_{4},~ & v_{11} & =-v_{5},~ & v_{12} & =-v_{6}, 
\end{alignat} 
and also for their images when applying the mappings $\tA$ and $\tB$:
\begin{equation}\label{eq:ab} 
  a_{i}=\tA v_{i},\quad b_{i}=\tB v_{i},\qquad i=1,2,\ldots,12. 
\end{equation} 

Finally, denote by $S$ the dodecagon with vertices $v_{i}$. Note that the vertices of~$S$ are divided into four groups of points, which are cyclically transformed into each other by applying suitable matrix~$\tA$ or~$\tB$:
\begin{gather*} 
  v_{1}\stackrel{\tA}{\longrightarrow}v_{9}\stackrel{\tA}{\longrightarrow}v_{5}\stackrel{\tB}{\longrightarrow}v_{1},\quad v_{7}\stackrel{\tA}{\longrightarrow}v_{3}\stackrel{\tA}{\longrightarrow}v_{11}\stackrel{\tB}{\longrightarrow}v_{7},\\ v_{2}\stackrel{\tA}{\longrightarrow}v_{10}\stackrel{\tB}{\longrightarrow}v_{6}\stackrel{\tB}{\longrightarrow}v_{2},\quad v_{8}\stackrel{\tA}{\longrightarrow}v_{4}\stackrel{\tB}{\longrightarrow}v_{12}\stackrel{\tB}{\longrightarrow}v_{8}. 
\end{gather*} 

\begin{remark}\label{rem:s-expl}
At first glance, the way the points $v_{i}$ are chosen looks rather strange. The explanation is that in the examples of constructing the polygon $S$ for the case $\varkappa=1.331$ and some randomly selected values of $\mu$ shown in Figure~\ref{F:0}, it was precisely this choice of the points $v_{i}$ that ensured their enumeration in ascending order of indices when traversing the vertices of the polygon $S$ clockwise around the origin.

In Section~\ref{app:rows-dispose} it will be shown that for any pair of parameters $\varkappa>1$ and $\mu>0$ no two rays of the form $R(v_{i})=\{x\in\mathbb{R}^{2}: x=t
v_{i},~t\ge0\}$, $i=1,2,\ldots,12$, coincide with each other. Moreover, the cyclic order of the points $v_{1},v_{2},\ldots,v_{12},v_{1}$ when traversing the vertices of the polygon $S$ clockwise around the origin remains the same for an arbitrary choice of the parameters $\varkappa>1$ and $\mu>0$ used in constructing the polygon~$S$.
\end{remark} 

It will be convenient for us to express the definitions of the points $v_{i}, a_{i}$ and $b_{i}$ with indices $i=1,\ldots,6$, as functions of the vectors $v$ and $w$:
\begin{equation}\label{eq:sab}%
 \left.
 \begin{aligned}
 v_{1}&=\mu v,\quad&a_{1}&=\tA v_{1}=v_{9},\quad&b_{1}&=\tB v_{1}=\mu\tB v,\\
 v_{2}&=w,&a_{2}&=\tA v_{2}=v_{10},\quad&b_{2}&=\tB v_{2}=\frac{1}{\lambda}v_{10},\\
 v_{3}&=-\tA v_{1}=-\mu\tA v,\quad&a_{3}&=\tA v_{3}=v_{11},\quad&b_{3}&=\tB v_{3}=-\mu\tB\tA v,\\
 v_{4}&=-\tA v_{2}=-\tA w,\quad&a_{4}&=\tA v_{4}=-\tA\tA w,\quad&b_{4}&=\tB v_{4}=v_{12},\\
 v_{5}&=-\tA v_{3}=\mu\tA\tA v\quad&a_{5}&=\tA v_{5}=\frac{1}{\lambda}v_{1},\quad&b_{5}&=\tB v_{5}=v_{1},\\
 v_{6}&=-\tB v_{4}=\tB\tA w,\quad&a_{6}&=\tA v_{6}=\tA\tB\tA w,\quad&b_{6}&=\tB v_{6}=v_{2}. 
 \end{aligned}\right\} 
\end{equation}%
Expressions for $v_{1},\ldots,v_{6}$ in~\eqref{eq:sab} follow from their definition, and expressions for $a_{1}, a_{2}, a_{3}, a_ {5}, b_{2}, b_{4}, b_{5}$, and $b_{6}$ follow from the chains of equalities given below, which in turn are a consequence of relations~\eqref{eq:tilde-aaa} and~\eqref{eq:eigen}:%
\begin{align*} 
 a_{1} & =\tA v_{1}=-v_{3}=v_{9},\\
 a_{2} & =\tA v_{2}=-v_{4}=v_{10},\\
 a_{3} & =\tA v_{3}=-v_{5}=v_{11},\\
 a_{5} & =\tA v_{5}=-\tA\tA v_{3}=\tA\tA\tA v_{1}=\frac{1}{\lambda}v_{1},\\
 b_{2} & =\tB v_{2}=\tB\tB\tB\tA v_{2}=\frac{1}{\lambda}\tA v_{2} =-\frac{1}{\lambda}v_{4}=\frac{1}{\lambda}v_{10},\\
 b_{4} & =\tB v_{4}=-v_{6}=v_{12},\\
 b_{5} & =\tB v_{5}=-\tB\tA v_{3}=\tB\tA\tA v_{1}=v_{1},\\
 b_{6} & =\tB v_{6}=-\tB\tB v_{4}=\tB\tB\tA v_{2}=v_{2}. 
\end{align*} 

\begin{remark}\label{rem:relations} 
Relations~\eqref{eq:sab} do not depend on the specific form of the matrices~$A$ and~$B$ or the choice of the eigenvectors $v$ and $w$ of the matrices $BAA$ and $BBA$. For their validity, it is only important that relations~\eqref{eq:tilde-aaa} be satisfied.
\end{remark} 

\begin{remark}\label{rem:kappa}
In all our constructions we further use the parameter $\varkappa=1.331=1.1^{3}$. This choice of $\varkappa$ is optional; it is due only to the fact that in this case, according to~\eqref{eq:AB2} and~\eqref{eq:def-tilde-ab}, both the matrices $A$ and $B$ as well as the matrices $\tA$ and $\tB$ turn out to be rational, which is nice in itself.
\end{remark} 

\begin{remark}\label{rem:badcases}
Figure~\ref{F:0} shows two examples of constructing dodecagons $S$ for 
$\varkappa=1.331$ and two random values of the parameter $\mu$. 
Both of these examples were ``unsuccessful'' from the point of view of 
wanting to construct the norm required in~\eqref{eq:tAnorm}, and demonstrate 
the typical difficulties encountered in constructing such a norm. Thus, in Figure~\ref{F:0a} the constructed dodecagon $S$ is not convex, and therefore cannot be the boundary of the unit ball for any norm. The polygon $S$ in Figure~\ref{F:0b}, although convex, does not satisfy the relation $\tB S\subseteq S$.
\end{remark} 

\begin{figure}[htbp!] 
  \centering \mbox{}\hfill\subcaptionbox{Case $\varkappa=1.331$ and $\mu=1.04$\label{F:0a}} {\includegraphics*[width=0.49\textwidth]{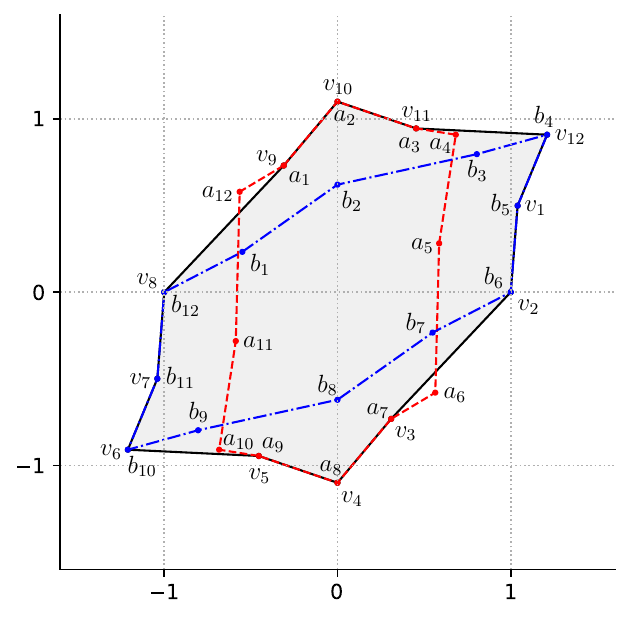}} \hfill\subcaptionbox{Case $\varkappa=1.331$ and $\mu=1.36$\label{F:0b}} {\includegraphics*[width=0.49\textwidth]{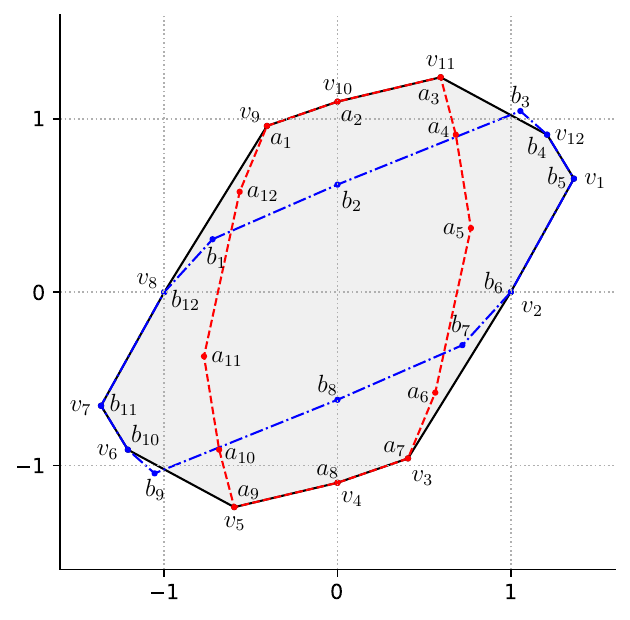}} \hfill\mbox{}\caption{Polygon $S$ (solid line, gray background) and its images $\tA S$ (dashed line) and $\tB S$ (dash-dotted line)}\label{F:0} 
\end{figure} 

\subsection{Selection of scaling parameter}\label{ssec:mu}
As follows from Remark~\ref{rem:badcases}, in order for the polygon $S$ to go from a ``candidate to be the unit sphere'' of some norm to a true unit sphere of the norm required in~\eqref{eq:tAnorm}, it is necessary to select the scaling parameter $\mu$ in such a way that the following requirements are met: 
\begin{enumerate}[(i)] 
  \item\label{item1} the polygon $S$ must be convex; 
  \item\label{item2} the inclusions $a_{i}=\tA v_{i}\subseteq S$, $b_{i}=\tB v_{i}\subseteq S$ must be satisfied for $i=1,2,\ldots,12$ or, equivalently, since the polygon $S$ is balanced, that the inclusions ${a_{i},b_{i}\in S}$ must be satisfied for $i=1,2,\ldots,6$.
\end{enumerate} 

\begin{remark}\label{rem:redundancy}
Condition~\ref{item1} is actually redundant and does not require checking. 
Indeed, suppose that the polygon $S$ with vertices $v_{i}$, 
$i=1,2,\ldots,12$, satisfies condition~\ref{item2}, but it is not convex. 
Taking then the convex hull $\tilde{S}=\co S$ of the polygon~$S$, we find 
that for $\tilde{S}$ both conditions~\ref{item1} and~\ref{item2} will 
already be satisfied, and, as a consequence, the inclusions 
$\tA\tilde{S}\subseteq\tilde{S}$ and $\tB\tilde{S} \subseteq\tilde{S}$ hold, too.
\end{remark} 

Taking this remark into account, to complete the construction of the norm 
required in~\eqref{eq:tAnorm}, it is enough for us to prove only the 
existence of such a value of the parameter $\mu$ that would ensure the 
fulfillment of condition~\ref{item2}. This task is simplified by the fact 
that due to~\eqref{eq:sab} the inclusions
\[
 a_{1},a_{2},a_{3},a_{5}\in S, \quad b_{2},b_{4},b_{5},b_{6}\in S
\]
hold automatically for all values of $\varkappa>1$ and $\mu>0$. 
Therefore, we only need to find such values of the parameter $\mu$ for which only $4$ ``nonobvious'' inclusions $a_{4},a_{6},b_{1},b_{3}\in S$ will be satisfied, which are more convenient for us to replace with equivalent (due to the equality $b_{1}=-b_{7}$) ones:
\begin{equation}\label{eq:inS} 
  a_{4},a_{6},b_{3},b_{7}\in S. 
\end{equation}
The problem is made even more concrete if we note that inclusions~\eqref{eq:inS} are satisfied if and only if each of the points $z=a_{4},a_{6},b_{3},b_{7}$ belongs to one of the triangles
\[
 \triangle xy0 :=\{sx+ty:~ s,t\ge 0,~ s+t\le1\}
\]
with vertices ${x,y,0}$, where $x=v_{i}, y=v_{i+1}$, $i=1,2,\ldots,12$ (assuming that ${v_{ 13}=v_{1}}$); see Figure~\ref{F:0}.

It will be convenient for us to check whether a certain point $z$ belongs to the triangle $\triangle xy0$ in two steps: first check whether the point $z$ belongs to the sector
\[
 S(x,y):=\{sx+ty:~ s,t\ge 0\},
\]
containing the triangle $\triangle xy0$, and only then check that, under the condition $z\in S(x,y)$, the point $z$ belongs to the triangle $\triangle xy0$. We present the corresponding conditions in Lemma~\ref{lem:in-triangle} in a form convenient for further application. But first, recall that the sector $S(x,y)$ does not degenerate into a ray, and the triangle $\triangle xy0$ does not degenerate into a point or a straight line segment if and only if the vectors~$x$ and~$y$ are linearly independent, which is equivalent to the inequality
\begin{equation}\label{eq:Lcond}
 (x,Ty)\neq0,\quad\text{where}\quad T=\begin{pmatrix*}[r] 0&-1\\ 1&0
 \end{pmatrix*},
\end{equation}
and $(\cdot,\cdot)$ is the Euclidean scalar product in $\mathbb{R}^{2}$.
\begin{lemma}\label{lem:in-triangle}
Let $x,y,z\in\mathbb{R}^{2}$ be a triple of points, the first two of which,~$x$ and~$y$, satisfy condition~\eqref{eq:Lcond}. Then $z\in S(x,y)$ if and only if
 \begin{equation}\label{eq:insect}
 s(x,y,z):=\frac{(z,Ty)}{(x,Ty)}\ge0,\quad t(x,y,z):=\frac{(z,Tx)} {(y,Tx)}\ge0.
 \end{equation}
And if inequalities~\eqref{eq:insect} are met, the inclusion $z\in\triangle xy0$ holds if and only if
 \begin{equation}\label{eq:intriag}
 h(x,y,z):=\frac{(y-x,Tz)}{(y,Tx)}\le 1.
 \end{equation}
\end{lemma} 
\begin{proof}
The vectors $x$ and $y$, subject to condition~\eqref{eq:Lcond}, are linearly independent and therefore form a basis in $\mathbb{R}^{2}$. Then the vector $z$ can be uniquely represented in the form
\begin{equation}\label{eq:sigma-tau}
 z=s x + t y,\qquad s,t\in\mathbb{R},\quad z\in S(x,y),
\end{equation}
if and only if
\begin{equation}\label{eq:st-ineq}
 s\ge0,\quad t\ge 0,
\end{equation}
and $z\in\triangle xy0$ if additionally
\begin{equation}\label{eq:st-ineq2}
 s+t\le1.
\end{equation}

Multiplying the equality in~\eqref{eq:sigma-tau} scalarly from the right first by $Ty$ and then by $Tx$, and taking into account that $(y,Ty)=(x,Tx)=0$ in the resulting expressions, and also $(y,Tx)=-(x,Ty)\neq0$ due to~\eqref{eq:Lcond}, we obtain that the quantities $s$, $t$, and $s+t$ are determined by the equalities
\[
 s=s(x,y,z),\quad t=t(x,y,z),\quad s+t=h(x,y,z).
\]
From here, by virtue of~\eqref{eq:st-ineq} and \eqref{eq:st-ineq2},  the relations~\eqref{eq:insect} and~\eqref{eq:intriag} follow, respectively. The lemma is proven.
\end{proof} 

We begin the proof of inclusions~\eqref{eq:inS} for some $\mu>0$ by noting that these inclusions can take place if and only if each of the points $a_{4},a_{6},b_ {3},b_{7}$ belongs to some triangle $\triangle xy0$ with neighboring vertices $x=v_{i}, y=v_{i+1}$, $i=1,2,\ldots,12$ (assuming that $v_{13}=v_{1}$). In order not to check all $12$ variants of pairs of neighboring vertices for each of the points $a_{4},a_{6},b_{3},b_{7}$, we turn to Figure~\ref{F:0}, which suggests that at least for $\varkappa=1.331$ the triangles to which the points $a_{4},a_{6},b_{3},b_{7}$ can belong are as follows: $\triangle v_{2} v_{3}0$ and $\triangle v_{11}v_{12}0$.

As shown in Section~\ref{app:rows-dispose} by direct computation, the vertex pairs $v_{2},v_{3}$ and $v_{11},v_{12}$ satisfy~\eqref{eq:Lcond}:
\[
 (v_{2},Tv_{3}) = (v_{11},Tv_{12}) =\frac{\varkappa^{7/3}\mu}{\varkappa^2+1}\neq 0, \qquad \forall~\varkappa>1,\mu>0
\]
which makes it possible to use Lemma~\ref{lem:in-triangle} for further research.

Simple calculations show that
\begin{align*} 
 s(v_{11},v_{12},a_{4}) & =\frac{\varkappa^4-1}{\varkappa^4\mu}>0, & t(v_{11},v_{12},a_{4}) & =\frac{1}{\varkappa^4}>0, \\ s(v_{2},v_{3},a_{6}) & =\frac{1}{\varkappa^4}>0, & t(v_{2},v_{3},a_{6}) & =\frac{\varkappa^4-1}{\varkappa^{10/3}\mu}>0, \\ s(v_{11},v_{12},b_{3}) & =\frac{1}{\varkappa^4}>0, & t(v_{11},v_{12},b_{3}) & =\frac{(\varkappa^6-1)\mu}{\varkappa^4(\varkappa^2+1)}>0, \\ s(v_{2},v_{3},b_{7}) & =\frac{(\varkappa^6-1)\mu }{\varkappa^{14/3} (\varkappa^2+1)}>0, & t(v_{2},v_{3},b_{3}) & =\frac{1}{\varkappa^4}>0, 
\end{align*} 
and these inequalities are true not only for $\varkappa=1.331$, but also for all $\varkappa>1$ and $\mu>0$. In this case, by virtue of the first statement of Lemma~\ref{lem:in-triangle}, inclusions
\[ 
  a_{4},b_{3}\in S(v_{11},v_{12}),\quad a_{6},b_{7}\in S(v_{2},v_{3}) 
\] 
hold. These inclusions are also true for all $\varkappa>1$ and $\mu>0$, which excludes the possibility that the points $ a_{4},a_{6},b_{3},b_{7}$ belong to any other sectors, and thereby relieves us of the need to check all $12$ variants of conditions~\eqref{eq:Lcond} for each of the points $a_{4},a_{6},b_{3},b_{7}$.

All we have to do now is write down conditions~\eqref{eq:intriag} from Lemma~\ref{lem:in-triangle} for all variants of the obtained vertices $v$ and points $a$ and $b$: 
\begin{align*} 
 h(v_{11},v_{12},a_{4}) & =\frac{\varkappa^4+\mu-1} {\varkappa^4\mu}\le1 & \Leftrightarrow & & \mu & \ge\mu_{0}(\varkappa):=1, \\ h(v_{2},v_{3},a_{6}) & =\frac{\varkappa^{2/3}(\varkappa^4-1)+\mu} {\varkappa^4 \mu}\le1 & \Leftrightarrow & & \mu & \ge\mu_{1}(\varkappa):=\varkappa^{2/3}, \\ h(v_{11},v_{12},b_{3}) & =\frac{(\varkappa^6-1)\mu+\varkappa^2+1} {\varkappa^4(\varkappa^2+1)}\le1 & \Leftrightarrow & & \mu & \le\mu_{2}(\varkappa):=\frac{(\varkappa^2+1)^2} {\varkappa^4+\varkappa^2+1}, \\ h(v_{2},v_{3},b_{7}) & =\frac{(\varkappa^6-1)\mu+\varkappa^{2/3}(\varkappa^2+1)} {\varkappa^{14/3}(\varkappa^2+1)}\le1 & \Leftrightarrow & & \mu & \le\mu_{3}(\varkappa):=\frac{\varkappa^{2/3}(\varkappa^2+1)^2} {\varkappa^4+\varkappa^2+1}. 
\end{align*} 
Here conditions~\eqref{eq:intriag} from Lemma~\ref{lem:in-triangle} are written out in the left column, and their reformulations in terms of the parameter $\mu>0$ are in the right column. As is easy to see, then
\[ 
  \mu_{0}(\varkappa)<\mu_{1}(\varkappa),\quad \mu_{2}(\varkappa)<\mu_{3}(\varkappa),\qquad\forall~ \varkappa>1 
\] 
and therefore $\mu$ must actually satisfy only two inequalities:
\begin{equation}\label{eq:allowed-mu}
 \mu_{1}(\varkappa)\le\mu\le \mu_{2}(\varkappa),\quad\text{where}\quad\varkappa>1.
\end{equation}

Let us now note that the inequality $\mu_{1}(\varkappa)\le\mu_{2}(\varkappa)$ holds 
only for
\begin{equation}\label{eq:allowed-varkappa} 
  \varkappa\in(1,\varkappa_{\text{max}}], 
\end{equation} 
where the value $\varkappa_{\text{max}}$ is uniquely determined from the equation $\mu_{1}(\varkappa)=\mu_{2}(\varkappa)$ and its approximate value is $\varkappa_{\text{max}}\approx 1.447892$. Therefore, inequalities~\eqref{eq:allowed-mu} can be satisfied by choosing $\mu$ only for those values of $\varkappa$ that satisfy  inclusion~\eqref{eq:allowed-varkappa}.

In particular, for $\varkappa=1.331$ we have the following numerical relations:
\begin{equation}\label{eq:mu-range} 
  \mu_{0} = 1 < \mu_{1} = 1.21 \le\mu\le\mu_{2} \approx 1.299757 < \mu_{3}\approx 1.572706. 
\end{equation} 
The corresponding set $S$, as well as its images $\tA S$ and $\tB S$, are depicted in  Figure~\ref{F:1}.
\begin{figure}[htbp!] 
  \centering 
  \includegraphics*[width=0.6\textwidth]{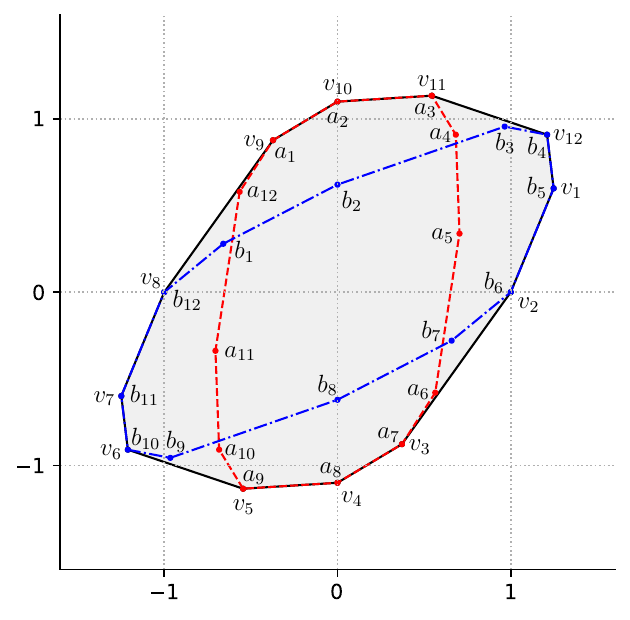} 
  \caption{Polygon $S$ (solid line, gray background) and its images $\tA S$ (dashed line) and $\tB S$ (dash-dotted line); $\varkappa=1.331$, $\mu=1.25$}\label{F:1} 
\end{figure} 

According to Remark~\ref{rem:redundancy} there is no need to prove the convexity of the polygon~$S$. However, in Section~\ref{app:convex}, for the sake of completeness, it is shown that for values of~$\varkappa$ and~$\mu$ satisfying~\eqref{eq:allowed-mu} and~\eqref{eq:allowed-varkappa}, the set $S$ is actually convex; see Figure~\ref{F:1}.

Let us summarize the results of our constructions in the following theorem.
\begin{theorem}\label{thm:main}
Let $\setA=\{A,B\}$ be the set of matrices~\eqref{eq:AB2}, and let $S$ be the polygon with vertices $v_{i}$, $i=1,2,\ldots,12$, defined by~\eqref{eq:s1-s6} and~\eqref{eq:s7-s12}. Then, for values of the parameters $\varkappa$ and $\mu$ satisfying~\eqref{eq:allowed-mu} and~\eqref{eq:allowed-varkappa}, the inclusions $a_{i}=\tA v_{i}\subseteq S$, $b_{i}=\tB v_{i}\subseteq S$ hold for $i=1,2,\ldots,12$ and, consequently, the set of matrices $\setA$ has the norm with the unit ball $S$ required in~\eqref{eq:tAnorm}.
\end{theorem} 

\section{Proofs and auxiliary relations}\label{sec:proofs} 

\subsection{Proof of Theorem~\ref{thm:smp}}\label{sec:proof-smp} 
Let $\sigma(X)$ denote the spectrum of a square matrix~$X$. If a matrix $X$ is a product of matrices 
\begin{equation}\label{eq:prodX} 
  X=X_{k}X_{k-1}\cdots X_{1},\qquad X_{i}\in\setA\quad\text{for}\quad i=1,2,\ldots,k, 
\end{equation} 
then through $\#_{A}(X)$ we will denote the number of factors $X_{i}$ of the form $A$ in the product~\eqref{eq:prodX}, and through $\#_{B}(X)$ the number of factors $X_{i}$ of the form $B$ in~\eqref{eq:prodX}.

As noted in Section~\ref{sec:example}, the mapping $\tau$ is multiplicative, and therefore, for the matrix $M$ from the conditions of the theorem, the equality
\begin{equation}\label{eq:Mprod} 
  \tau(M)=\tau(M_{k})\tau(M_{k-1})\cdots\tau(M_{1}) 
\end{equation} 
holds. Then $\sigma(M)=\sigma(\tau(M))$, which means that the spectral radii of the matrices~$M$ and $\tau(M)$ coincide. Since both matrices~$M$ and $\tau(M)$ are represented by products with the same number of matrix factors from the set~$\setA$, and the matrix~$M$, by the conditions of the theorem, is spectrum maximizing, then the matrix~$\tau( M)$ is also spectrum maximizing.

Finally, note that by definition the mapping $\tau$, when applied, swaps the matrices~$A$ and $B$. Then
\[
 \#_{A}(\tau(M))=\#_{B}(M),\quad \#_{B}(\tau(M))=\#_{A}(M).
\]
For an odd number of factors $M_{i}$ in~\eqref{eq:prodM}, in this case $\#_{B}(M)\neq\#_{A}(M)$, and therefore
\[
 \#_{A}(\tau(M))=\#_{B}(M)\neq\#_{A}(M).
\]
Consequently, the matrices $M$ and $\tau(M)$ turn out to be different in the sense that neither of them is a cyclic permutation of the factors of the other. The theorem is proven.

\subsection{Location of the vertices of the polygon $S$}\label{app:rows-dispose} 

Let
\[
  R(x)=\{y\in\mathbb{R}^{2}: y=t x,~t\ge0\},\quad 
  L(x)=\{y\in\mathbb{R}^{2}: y=t x,~t\in\mathbb{R}\}
\]
denote the ray and the straight line, respectively, passing through the point $x\ne 0$ and the origin. Note that for nonzero points $x$ and $y$ the following condition of ``noncoincidence of the lines $L(x)$ and $L(y)$'' is true: 
\[ 
  L(x)\cap L(y)=\{0\}\quad\Longleftrightarrow\quad (x,Ty)\neq 0, 
\] 
where $(\cdot,\cdot)$ is the Euclidean scalar product in $\mathbb{R}^{2}$, and $T$ is the rotation matrix by the angle $\frac{\pi}{2}$ counterclockwise, defined by equality~\eqref{eq:Lcond}.

For each point $x\ne 0$ the set $\mathbb{R}^{2}\backslash L(x)$ consists of two open half-planes. Denote by
\[
\mathbb{R}^{2}_{+}(x)=\{y\in\mathbb{R}^{2}: (x,Ty)> 0\}
\]
the half-plane formed by the vectors $y\in\mathbb{R}^{2}$ obtained from $x$ by clockwise rotation through some angle $\varphi\in(0,\pi)$ and subsequent suitable compression or expansion.

Direct calculations show that for $\varkappa>1$, $\mu>0$, and $\lambda=\varkappa^{2}$ the following relations hold:
{\allowdisplaybreaks%
\begin{align*} 
 (v_{1},Tv_{2}) & =\frac{\varkappa\mu}{\varkappa^2+1}> 0, & (v_{1},Tv_{3}) & =\frac{\varkappa^{1/3}(\varkappa^4+\varkappa^2+1)\mu^2} {(\varkappa^2+1)^2}> 0, \\ (v_{1},Tv_{4}) & =\varkappa^{1/3}\mu> 0, & (v_{1},Tv_{5}) & =\frac{(\varkappa^4+\varkappa^2+1)\mu^2} {\varkappa^{1/3}(\varkappa^2+1)^2}> 0, \\ (v_{1},Tv_{6}) & =\frac{\mu}{\varkappa^{1/3}(\varkappa^2+1)}> 0, & (v_{2},Tv_{3}) & =\frac{\varkappa^{7/3}\mu}{\varkappa^2+1}> 0, \\ (v_{2},Tv_{4}) & =\varkappa^{1/3}> 0, & (v_{2},Tv_{5}) & =\frac{\mu}{\varkappa^{1/3}}> 0, \\ (v_{2},Tv_{6}) & =\frac{1}{\varkappa^{1/3}}> 0, & (v_{3},Tv_{4}) & =\frac{\mu}{\varkappa^{1/3}(\varkappa^2+1)}> 0, \\ (v_{3},Tv_{5}) & =\frac{(\varkappa^4+\varkappa^2+1)\mu^2} {\varkappa(\varkappa^2+1)^2}> 0, & (v_{3},Tv_{6}) & =\frac{(\varkappa^4+1)\mu} {\varkappa(\varkappa^2+1)}> 0, \\ (v_{4},Tv_{5}) & =\frac{\varkappa\mu}{\varkappa^2+1}> 0, & (v_{4},Tv_{6}) & =\varkappa> 0, \\ (v_{5},Tv_{6}) & =\frac{\varkappa^{7/3}\mu}{\varkappa^2+1}> 0. & 
\end{align*}%
}%
Due to these relations
\[
 L(v_{i})\cap L(v_{j})=\{0\}\qquad\text{for}\quad i,j=1,2,\ldots,6,~i\neq j.
\]
And since $v_{i+6}=-v_{i}$ for $i=1,2,\ldots,6$, then the previous relations show that the points $v_{1},v_{2},\ldots ,v_{12}$ actually generate only $6$ lines, $L(v_{1}),L(v_{2}),\ldots,L(v_{6})$, which are different for all $\varkappa>1$ and $\mu>0$.

The obtained expressions for $(v_{i},Tv_{j})$ show that
\begin{equation}\label{eq:vivj}
(v_{i},Tv_{j})>0\quad\text{for}\quad i=1,2,\ldots,5,~j=2,3,\ldots,5,~ i<j.
\end{equation}
Consequently, all vertices $v_{2},v_{3},\ldots,v_{6}$ belong to the open half-plane
$\mathbb{R}^{2}_{+}(v_{1})$:
\[
v_{2},v_{3},\ldots,v_{6}\in \mathbb{R}^{2}_{+}(v_{1}).
\]
It also follows from~\eqref{eq:vivj} that for $i=2,3,\ldots,5$ each vector $v_{i+1}$ is obtained from $v_{i}$ by a clockwise rotation by some angle $\varphi\in(0,\pi)$ and subsequent appropriate compression or expansion. This means that, when going clockwise around the origin, the vertices $v_{1},v_{2},\ldots,v_{6}$ of the polygon $S$ appear in the order of their indexing. Considering that the vertices $v_{7},v_{8},\ldots,v_{12}$ are obtained from $v_{1},v_{2},\ldots,v_{6}$ by mirroring relative to the origin, we can repeat the same conclusion regarding their order that was made for the vertices $v_{1},v_{2},\ldots,v_{6}$. And then a similar conclusion can be made for the entire set of vertices $v_{1},v_{2},\ldots,v_{12}$, and this conclusion does not depend on the choice of parameters $\varkappa>1$ and $\mu>0$.

\subsection{Convexity of the polygon $S$}\label{app:convex}
As shown in the previous section, the vertices $v_{1},v_{2},\ldots,v_{6}$ of the polygon $S$ appear in the order of their indexing when going clockwise around the origin. Therefore, to prove the convexity of the polygon $S$, it suffices to show that each of the vertices $v_{i}$, $i=1,2,\ldots,12$, does not belong to the interior of the triangle $\triangle
v_{i-1}v_{i+1}0$, where it is assumed that $v_{0}=v_{12}$ and $v_{13}=v_{1}$. By Lemma~\ref{lem:in-triangle} the corresponding ``nonbelonging'' conditions are equivalent to the inequalities
\[
 h(v_{i-1},v_{i+1},v_{i})\ge 1,\qquad i=1,2,\ldots,6.
\]
Let us write down the corresponding conditions, as well as their reformulations in terms of the parameter~$\mu$: 
{\allowdisplaybreaks%
\begin{align*} 
 h(v_{12},v_{2},v_{1}) & =\frac{(\varkappa^{4/3}+1)\mu} {\varkappa^2+1}\ge1 & \Leftrightarrow & & \mu & \ge\omega_{1}(\varkappa):=\frac{\varkappa^2+1}{\varkappa^{4/3}+1}, \\ h(v_{1},v_{3},v_{2}) & =\frac{(\varkappa^2+1)(\varkappa^2+\varkappa^{2/3})} {(\varkappa^4+\varkappa^2+1)\mu}\ge1 & \Leftrightarrow & & \mu & \le\omega_{2}(\varkappa):=\frac{(\varkappa^2+1)(\varkappa^2+\varkappa^{2/3})} {\varkappa^4+\varkappa^2+1}, \\ h(v_{2},v_{4},v_{3}) & =\frac{(\varkappa^{8/3}+1)\mu} {\varkappa^{2/3} (\varkappa^2+1)}\ge1 & \Leftrightarrow & & \mu & \ge\omega_{3}(\varkappa):=\frac{\varkappa^{2/3} (\varkappa^2+1)} {\varkappa^{8/3}+1}, \\ h(v_{3},v_{5},v_{4}) & =\frac{(\varkappa^2+1)(\varkappa^2+\varkappa^{2/3})} {(\varkappa^4+\varkappa^2+1)\mu}\ge1 & \Leftrightarrow & & \mu & \le\omega_{4}(\varkappa):=\frac{(\varkappa^2+1)(\varkappa^2+\varkappa^{2/3})} {\varkappa^4+\varkappa^2+1}, \\ h(v_{4},v_{6},v_{5}) & =\frac{(\varkappa^{4/3}+1)\mu} {\varkappa^2+1}\ge1 & \Leftrightarrow & & \mu & \ge\omega_{5}(\varkappa):=\frac{\varkappa^2+1}{\varkappa^{4/3}+1}, \\ h(v_{5},v_{7},v_{6}) & =\frac{(\varkappa^2+1)(\varkappa^{8/3}+1)} {(\varkappa^4+\varkappa^2+1)\mu}\ge1 & \Leftrightarrow & & \mu & \le\omega_{6}(\varkappa):=\frac{(\varkappa^2+1)(\varkappa^{8/3}+1)} {\varkappa^4+\varkappa^2+1}. 
\end{align*}
} 

It is easy to check that for all $\varkappa>1$ and $\mu>0$ the following inequalities hold:
\[
\omega_{3}(\varkappa)\le\omega_{1}(\varkappa)=\omega_{5}(\varkappa)\le\mu_{1}(\varkappa),\quad \mu_{2}( \varkappa)\le\omega_{2}(\varkappa)=\omega_{4}(\varkappa)\le\omega_{6}(\varkappa).
\]
Then, for values of the parameters $\varkappa$ and $\mu$ satisfying~\eqref{eq:allowed-mu} and~\eqref{eq:allowed-varkappa}, the inequalities
\[ 
  \omega_{3}(\varkappa)\le\omega_{1}(\varkappa)=\omega_{5}(\varkappa)\le\mu_{1}(\varkappa)\le \mu\le\mu_{2}(\varkappa)\le\omega_{2}(\varkappa)=\omega_{4}(\varkappa)\le\omega_{6}(\varkappa) 
\] 
are also satisfied and, therefore, for these values of the parameters $\varkappa$ and $\mu$, the polygon $S$ is convex.

\section{Notes and comments}\label{sec:comments} 

\begin{remark}\label{rem:notJSRtoolbox}
It would seem that in order to construct the required norm it would be natural to use the numerical implementation of the ``polytope algorithm''~\cite{GugProt:FCM13, JSRToolbox}, which allows, according to the authors, ``an accurate calculation of the joint spectral characteristics of matrices'' and to determine whether a specific matrix product is spectrum maximizing.

Unfortunately, our attempt to use the \texttt{JSR~Toolbox}~\cite{JSRToolbox} for this purpose was unsuccessful: it turned out that the version of this package available in MATLAB Central is outdated and incompatible with MATLAB releases R2023a and later due to changes in the format and parameters of calling the linear programming solver \texttt{linprog} from the \texttt{Optimization Toolbox} in these releases.

After our amateur ``correction'' of calling the \texttt{linprog} procedure in the script \texttt{\detokenize{jsr_norm_balancedRealPolytope.m}}\footnote{The related programs are available on the website \url{https://github.com/kozyakin/spectrum_maximizing_products}.} the functionality of \texttt{JSR~Toolbox} from a formal point of view has been ``restored.'' However, not being sure that our intervention did not affect the correct operation of the corresponding algorithm, when constructing the example in Section~\ref{sec:example} we decided to abandon the use of numerical methods and preferred to use an analytical approach instead.

However, it is worth noting that for the set of matrices in Example~\ref{ex:main}, the polygon~$S$ calculated using \texttt{JSR~Toolbox} coincided with the polygon~$S$ constructed in Section~\ref{sec:example} for $\mu=\mu_{2}$, i.e., at the maximum permissible value of the parameter $\mu$ according to~\eqref{eq:mu-range}.
\end{remark} 

\begin{remark}\label{rem:first-ex}
The scheme for constructing the required norm, applied in section~\ref{sec:example} for matrices from Example~\ref{ex:main}, turned out to be efficient when analyzing matrix products with matrices from Example~\ref{ex:alt}. The only difference was that the more familiar form of specifying rotation matrices turned out to be less suitable for symbolic manipulations in ``manual mode'' due to the appearance of many quadratic radicals in intermediate calculations. However, in this case too, with the help of the Wolfram~\emph{Mathematica} program, all the necessary transformations were completed. 

It turned out that in this case the norm required in~\eqref{eq:tAnorm} exists for
\[
\varkappa\in(1,\varkappa_{\text{max}}],\quad\text{where}\quad \varkappa_{\text{max}}\approx 1.528580.
\]
In particular, for $\varkappa=1.331$ the following relations were obtained:
\[ 
  \mu_{0} \approx 0.874539 < \mu_{1} \approx 1.032076 \le\mu\le\mu_{2} \approx 1.143460 < \mu_{3}\approx 1.349441. 
\] 
The corresponding set $S$, as well as its images $\tA S$ and $\tB S$, are shown in Figure~\ref{F:2}. 
\begin{figure}[htbp!] 
  \centering 
  \includegraphics*[width=0.6\textwidth]{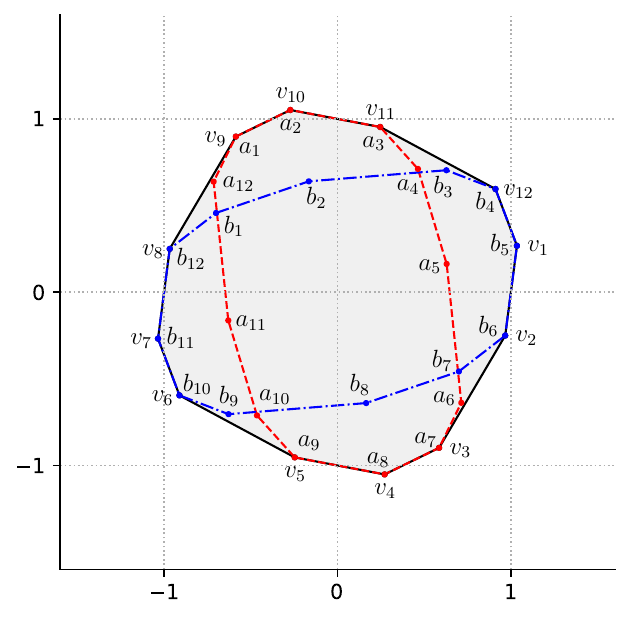} 
  \caption{Polygon $S$ from Example~\ref{ex:alt} (solid line, gray background) and its images $\tA S$ (dashed line) and $\tB S$ (dash-dotted line); $\varkappa=1.331$, $\mu=1.07$}\label{F:2} 
\end{figure} 
\end{remark} 

\begin{remark}\label{rem:github}
When performing symbolic transformations in Section~\ref{sec:example} we used the Wolfram~\emph{Mathematica} program. Visualization of polygons $S$ and their images $\tA S$ and $\tB S$ for Figures~\ref{F:0}--\ref{F:2} was carried out using Python programs. When constructing these figures, calculations were carried out at $\varkappa=1.331$. The reason for choosing such a ``strange'', at first glance, value for the parameter $\varkappa$ is explained in Remark~\ref{rem:kappa}. In preliminary test calculations, other values of the parameter were used, as well as programs based on a slight modification of the max-relaxation algorithm~\cite{Koz:DCDSB10} for the iterative construction of Barabanov norms.

All programs used in this work are available for download and free use from the website \url{https://github.com/kozyakin/spectrum_maximizing_products}.
\end{remark} 

\section{Conclusion}\label{S:conclusion}
The motivation for this work was the example constructed in~\cite{BochiLas:SAIMJMAA24} of a set of matrices $\{A,B\}$ that has two different (up to cyclic permutations of factors) spectrum maximizing products, $AABABB$ and $BBABAA$. The specific feature of this example (and other similar examples in~\cite{BochiLas:SAIMJMAA24}) was that the number of factors $A$ (or $B$) in both products turned out to be the same. In this regard, the question arose: are there pairs of spectrum maximizing products with different numbers of factors $A$ (or $B$) in these products? This paper gives a positive answer to this question. For this purpose, a method for constructing matrix sets $\{A,B\}$ is proposed, different from the method used in~\cite{BochiLas:SAIMJMAA24}, in which the number of spectrum maximizing products with an odd number of factors automatically turns out to be at least two, and the number of factors of the form $A$ (or $B$) in such products is different. Using this approach, we constructed an example of a set of $2{\times}2$ matrices $\{A,B\}$, for which the spectrum maximizing products have the form $BAA$ and $BBA$. Unlike~\cite{BochiLas:SAIMJMAA24}, where the construction of the corresponding examples essentially used the numerical implementation of the ``polytope algorithm''~\cite{GugProt:FCM13, JSRToolbox}, in this work, analytical methods were used to construct the examples.

\section*{Acknowledgments}

The author thanks the reviewers for several valuable comments and suggestions. Special thanks go to one of the reviewers who drew the author's attention to Friedland's theorem~\cite[Thm.~2.2]{Friedland:AM83}. Using this theorem together with the reviewer's detailed comments made it possible to significantly simplify the initial definition of $\tau$-permutability of a set of matrices and formulate a simple criterion for such permutability.


\providecommand{\KeyWords}[1]{#1}

\end{document}